\documentclass[11 pt]{article}

\usepackage[utf8]{inputenc}
\usepackage[english]{babel}   
\usepackage{graphicx}
\usepackage{lmodern}
\usepackage{amsmath}
\usepackage{tikz}
\usetikzlibrary{calc}
\usepackage{longtable}
\usepackage{epsfig}
\usepackage{color}
\usepackage{multicol}
\usepackage{setspace}
\usepackage[top=2cm,bottom=2.5cm,left=2.5cm,right=2.5cm]{geometry}
\usepackage{sectsty}
\usepackage{fancyhdr}
\usepackage{wasysym}
\usepackage[colorlinks=true]{hyperref}
\usepackage{subcaption}
\usetikzlibrary{patterns}
\usepackage{amsfonts}
\usepackage{amssymb} 
\usepackage{amsthm}

\newtheorem{theorem}{Theorem}

\newtheorem{lemma}[theorem]{Lemma}
\newtheorem{proposition}[theorem]{Proposition}
\newtheorem{remark}[theorem]{Remark}

\newcommand\R{\mathbb{R}}
\newcommand\C{\mathbb{C}}

\title{Computation of the deformation of rhombi-slit kirigami}
\date{}
\author{\begin{minipage}{\textwidth}\centering
		Frédéric Marazzato \\
		\small{Department of Mathematical Sciences, University of Nevada Las Vegas, Las Vegas, NV 89154-4020, USA}\\
		\small{\texttt{frederic.marazzato@unlv.edu}}
   \end{minipage}
   }
      
\begin{document}
\hypersetup{urlcolor=blue,linkcolor=red,citecolor=blue}

\maketitle

\begin{abstract}
Kirigami are part of the larger class of mechanical metamaterials, which exhibit exotic properties.
This article focuses on rhombi-slits, which is a specific type of kirigami.
A nonlinear kinematic model was previously proposed as a second order divergence-form PDE with a possibly degenerate, and sign-changing coefficient matrix.
We first propose to study the existence of solutions to a regularization of this equation by using the limiting absorption principle.
Then, we propose a finite element method to approximate the solutions to the regularized equation.
Finally, simulations are compared with experimental results.
\end{abstract}


 
\section{Introduction}
Kirigami is a variation of origami, the Japanese art of paper folding.
In kirigami, the paper is cut as well as being folded.
Origami and kirigami have been studied as concrete examples of mechanical metamaterials \cite{schenk2013geometry,boatti2017origami,
zheng2023modelling,wickeler2020novel}.
Mechanical metamaterials are solids with unusual mechanical properties that are generated by inner mechanisms of the system.
These materials have found applications for generating soft robots \cite{rafsanjani2019programming} or in
aerospace engineering \cite{morgan2016approach,manan}.
The mechanisms giving metamaterials their unusual properties are discrete by nature, and have been modeled as such \cite{DANG2022111224,FENG2020104018}.
Recently, important efforts have been devoted to proposing a continuous description of origami and kirigami-based metamaterials
\cite{nassar2017curvature,lebee2018fitting,
nassar2022strain,zheng2022continuum}.
The modeling process results in homogenized PDEs where the characteristic size of the mechanism is considered negligible with respects to the size of the whole structure.
Some first results regarding the existence of solutions to these nonlinear PDEs and their numerical approximation have been achieved in \cite{marazzato,marazzato2022mixed,zheng2022continuum}.

This paper focuses on kirigami patterns with parallelogram panels and rhombi slits as presented in \cite{zheng2022continuum}.
Figure \ref{fig:notation} shows one such kirigami pattern.
In this entire article, deformations remain planar.
$\xi_\mathrm{in}$ is defined as the slit opening in the undeformed configuration.
Let $\Omega \subset \mathbb{R}^2$ be an open bounded set.
$\xi :\Omega \to \mathbb R$ is the opening of the slit, and $\gamma :\Omega \to \mathbb R$ is the local rotation of a panel, as represented in Figure \ref{fig:notation}.
\begin{figure} [htb]

\begin{minipage}{0.4\textwidth}
   \centering
   \begin{tikzpicture}
   \pgfmathsetmacro{\l}{1}
   \fill (\l,0) -- (0,\l) -- (\l,2*\l) -- (2*\l,\l) -- cycle; 
   \fill (3*\l,0) -- (2*\l,\l) -- (3*\l,2*\l) -- (4*\l,\l) -- cycle; 
   \fill (3*\l,0) -- (2*\l,-\l) -- (3*\l,-2*\l) -- (4*\l,-\l) -- cycle; 
   \fill (\l,0) -- (0,-\l) -- (\l,-2*\l) -- (2*\l,-\l) -- cycle; 
   \draw[->] (1.5*\l,-0.5*\l) arc (-45:45:0.7);
   \node[right] at (1.75*\l,0) {$2\xi_{\mathrm{in}}$};
   \end{tikzpicture}
   
\end{minipage}
   
\begin{minipage}[b]{0.4\textwidth}
  
   \tikzset{
    square/.pic={
   \pgfmathsetmacro{\l}{1}
   \fill (\l,0) -- (0,\l) -- (\l,2*\l) -- (2*\l,\l) -- cycle; 
   } }
   
   \tikzset{
    test/.pic={
   \pic[rotate=20] at (0,0) {square};
    	\pic[rotate=-20] at (1.17,0.67) {square};
    \pic[rotate=-20] at (-0.7,-1.2) {square};
    \pic[rotate=20] at (1.87,-1.887) {square};
    \draw[red] (-.35,2.22) -- (3.41,2.22) -- (3.41,-1.55) -- (-.35,-1.55) -- cycle;
    \draw[red] (3.41,-1.55) -- (3.67,0);
    \draw[->,red] (3.6,-.5) arc (90:100:1);
    \node[red,right,rotate=-10] at (3.6,-.5) {$\gamma$};
    \draw[->] (1.08,0) arc (-40:40:.5);
    \node[preaction={fill, white, opacity=.6}, rotate=-10] at (2.1,0) {$2(\xi+\xi_\mathrm{in})$};
   } }

   \begin{tikzpicture}
   \pic[transform canvas={rotate=10},shift={(10,-.5)}] {test};
   \end{tikzpicture}
   
\end{minipage}  
   
   \caption{Notations.}
   \label{fig:notation}
\end{figure} 
$y_\mathrm{eff}: \Omega \to \mathbb R^2$ is the effective deformation, which  tracks the cell-averaged panel motions.
Using a coarse-graining technique, \cite{zheng2022continuum} has proposed that
\begin{equation}
\label{eq:yeff}
\nabla y_\mathrm{eff} = R(\gamma) A_\mathrm{eff}(\xi),
\end{equation}
where $R(\gamma)$ is the canonical $2\times2$ rotation matrix parametrized by the angle $\gamma$, and $A_\mathrm{eff}(\xi) \in \mathbb R^{2 \times 2}$ is called the shape tensor.
Note that the full gradient of $y_\mathrm{eff}$ appears in \eqref{eq:yeff} instead of only the symmetrized gradient as in small strain elasticity.
This can be concisely explained by the fact that the framework at play, in this paper, is that of finite strains and not of small strain elasticity.
See \cite{zheng2022continuum} for more details.

As the left-hand side of \eqref{eq:yeff} is a gradient, a necessary condition to have solutions of \eqref{eq:yeff} is that the curl of the right-hand side be zero.
In \cite{zheng2022continuum} it is shown that this necessary condition can be written as
\begin{equation}
\label{eq:gamma}
\nabla \gamma = \Gamma(\xi) \nabla \xi,
\end{equation}
where $\Gamma(\xi)$ is a $2\times2$ matrix, which depends on $A_\mathrm{eff}(\xi)$.
Similarly, a necessary condition for \eqref{eq:gamma} to be verified is
\begin{equation}
\label{eq:main}
-\mathrm{div}\left(R\left(\frac\pi2\right) \Gamma(\xi) \nabla \xi\right) = 0.
\end{equation}
The main goal of this paper is to compute approximate solutions to \eqref{eq:main} supplemented with appropriate Dirichlet and Neumann boundary conditions.
In \cite{zheng2022continuum}, the authors propose methods to compute solutions of \eqref{eq:main} only for two specific patterns, see Section \ref{sec:elliptic} and Section \ref{sec:hyperbolic} of this paper.
In the follow-up work \cite{zheng2023modelling}, the authors give an energy and compute its minimizers but they do not concurrently solve the kinematic constraint.
The present work proposes a way to approximate the solutions of \eqref{eq:main} in a more general case.

Let $B(\xi) := R\left(\frac\pi2 \right)\Gamma(\xi)$.
Studying the existence of solutions to \eqref{eq:main} presents two main difficulties.
The first is that $B(\xi)$ can degenerate for certain values of $\xi$.
The second is that the sign of $\mathrm{det}(B(\xi))$ can change, thus locally changing the type of \eqref{eq:main} from an elliptic to a hyperbolic PDE.
A similar problem is encountered with Tricomi and Keldysh equations, see \cite{chen2012mixed} for instance.
Regarding the issue of degeneracy, Muckenhopt weights \cite{turesson2000nonlinear} have been used to study the existence of solutions to degenerate elliptic equations, see \cite{di2008harnack}, for instance.
Muckenhopt weights can be used, for instance, to handle degeneracies localized at a point.
However, as in \cite{nicolopoulos2020}, the degeneracy we encounter is localized on a curve, and thus Muckenhopt weights cannot be used.
Regarding the issue of sign-change, the concept of $\mathtt{T}$-coercivity \cite{dhia2012t} has been used to study the existence of solutions for some linear equations \cite{dhia2012t,chesnel2013t}.
Unfortunately, this concept does not seem to be applicable here as the equation can be hyperbolic in a part of the domain.
Instead, we follow \cite{nicolopoulos2020,ciarlet2022mathematical} and use a limiting absorption principle to approximate solutions of \eqref{eq:main}.

Several numerical methods have been proposed to approximate solutions of sign-changing problems.
In \cite{chesnel2013t,chaumont2021generalized}, methods based on first-order Lagrange polynomials and $\mathtt{T}_h$-coercive meshes are devised.
In \cite{ciarlet2023optimal}, a method using ideas from optimal control was proposed.
However, these methods require the curve where \eqref{eq:main} changes type to be meshed exactly, which is not possible in our problem since the position of the interface is a priori unknown.
Therefore, we compute solutions of a problem regularized by adding a complex dissipation, as already performed in \cite{chesnel2013t}, which allows us to handle a non-scalar matrix $B(\xi)$ and to not have a priori knowledge of the interface.

The present article is structured as follows.
Section \ref{sec:continuous} focuses on proving the existence of solutions for a regularization of \eqref{eq:main} by using the the limiting absorption principle.
Section \ref{sec:discrete} proposes a numerical method to approximate the solutions of the  regularization of \eqref{eq:main}, supplemented by a convergence proof.
Finally, Section \ref{sec:test} compares a few numerical tests to experimental results.
Note that the limiting case when the dissipation parameter goes to zero is only studied from the numerical perspective.

\section{Approximate continuous problem}
\label{sec:continuous}
Let $\Omega \subset \mathbb{R}^2$ be an open bounded polygonal domain, that can be perfectly fitted by triangular meshes, with a Lipschitz boundary $\partial \Omega$.
The boundary is partitioned as $\partial \Omega = \partial \Omega_D \cup \partial \Omega_N$, where $\partial \Omega_D$ is relatively closed in $\partial \Omega$.
The exterior normal to $\partial \Omega$ is written as $n$.
A Dirichlet boundary condition $\xi_D \in H^\frac12(\partial \Omega_D ; \mathbb{R})$ is imposed strongly on $\partial \Omega_D$ and a Neumann boundary condition $g \in L^2(\partial \Omega_N;\R)$ is imposed weakly on $\partial \Omega_N$.
Let $V := H^1(\Omega;\mathbb{C})$ be a complex Sobolev space equipped with the usual Sobolev norm written as $\| \cdot \|_V$.
We define the solution set as $V_D := \{ \zeta \in V | \zeta = \xi_D \text{ on } \partial \Omega_D \}$, and $V_0 := \{ \zeta \in V | \zeta = 0 \text{ on } \partial \Omega_D \}$ as the associated homogeneous space.

\subsection{Rhombi-slits}
Following \cite{zheng2022continuum}, we focus in this paper on the specific case of rhombi slits.
Two geometric parameters $\alpha \le 0$, and $\beta  \ge 0$ are used to describe the geometry of a given unit cell.
Figure \ref{fig:sketch} shows two different unit cells.

\begin{figure}[!ht]
    \begin{subfigure}[b]{0.45\textwidth}
    \centering
        \includegraphics[scale=.5]{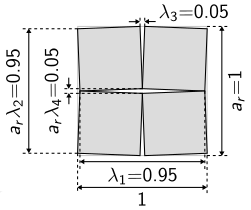}
		\caption{Auxetic kirigami.}
    	\label{fig:auxetic sketch}
    \end{subfigure}
    \hfill\begin{subfigure}[b]{0.45\textwidth}
    \centering
        \includegraphics[scale=.5]{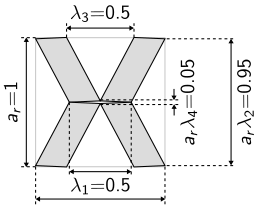}
		\caption{Non-auxetic kirigami.}
		\label{fig:non-auxetic sketch}
    \end{subfigure}
    \caption{Unit cells \cite{zheng2022continuum}.}
    \label{fig:sketch}
\end{figure}
Using the notations of Figure \ref{fig:sketch}, the geometric parameters are computed as
\[ \alpha = a_r(\lambda_4 - \lambda_2) \text{ and } \beta = \frac{\lambda_1 - \lambda_3}{a_r}. \]
Let $\xi \in L^2(\Omega;\R)$.
We define
\[ \mu_1(\xi) = \cos(\xi) - \alpha\sin(\xi), \quad \mu_2(\xi) = \cos(\xi) + \beta \sin(\xi), \] and \[ A_\mathrm{eff}(\xi) = \mu_1(\xi) e_1 \otimes e_1 + \mu_2(\xi) e_2 \otimes e_2. \]
Let $\Gamma(\xi) = \Gamma_{12}(\xi) e_1 \otimes e_2 + \Gamma_{12}(\xi) e_1 \otimes e_2$, where $(e_1,e_2)$ is the canonical basis of $\mathbb{R}^2$, and
\[ \Gamma_{12}(\xi) = -\frac{\mu_1'(\xi)}{\mu_2(\xi)}, \quad \Gamma_{21}(\xi) = \frac{\mu_2'(\xi)}{\mu_1(\xi)}. \]
Finally, one has
\[ B(\xi) = \begin{pmatrix}
-\Gamma_{21}(\xi) & 0 \\
0 & \Gamma_{12}(\xi) \\
\end{pmatrix}. \]
The strong formulation consists in searching for $\xi$ such that
\begin{equation}
\label{eq:strong form}
\left\{ \begin{aligned}
&-\mathrm{div}(B(\xi) \nabla \xi) = 0 \text{ a.e. in } \Omega, \\
& B(\xi) \nabla \xi \cdot n = g \text{ on } \partial \Omega_N, \\
& \xi = \xi_D \text{ on } \partial \Omega_D,
\end{aligned} \right.
\end{equation}

Let us show in practical cases the issues regarding sign-change and degeneracy mentioned above.
Figure \ref{fig:eigenvalues} reports the values of the eigenvalues of $B(\xi)$ for the patterns in Figure \ref{fig:sketch}.
\begin{figure}[!ht]
    \begin{subfigure}[b]{0.45\textwidth}
    \centering
        \includegraphics[scale=.5]{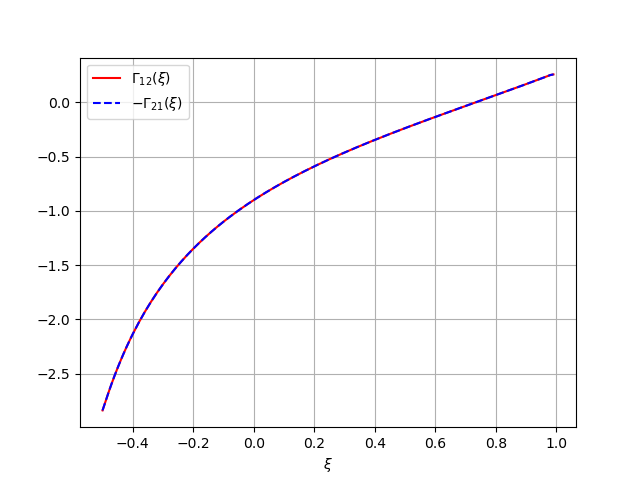}
		\caption{Auxetic kirigami.}
    	\label{fig:auxetic eigen}
    \end{subfigure}
    \hfill\begin{subfigure}[b]{0.45\textwidth}
    \centering
        \includegraphics[scale=.5]{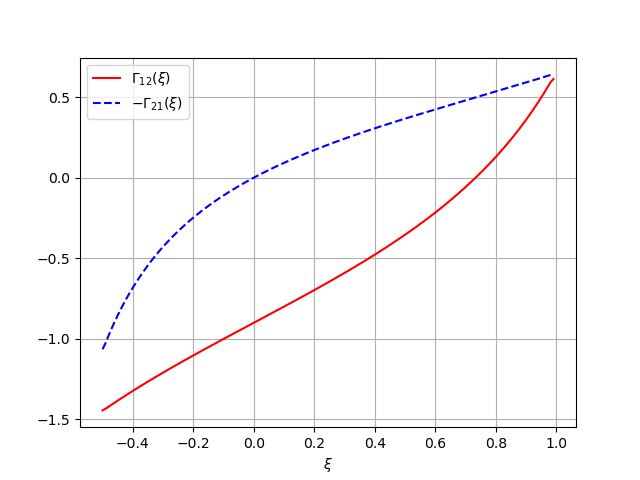}
		\caption{Non-auxetic kirigami.}
		\label{fig:non-auxetic eigen}
    \end{subfigure}
    \caption{Eigenvalues of $B(\xi)$.}
    \label{fig:eigenvalues}
\end{figure}
Figure \ref{fig:auxetic eigen} shows that $B(\xi)$ degenerates when $\xi \simeq 0.75$.
However, since $-\Gamma_{21}=\Gamma_{12}$ in the auxetic case, the pattern remains elliptic above or below that critical value.
The situation is different for Figure \ref{fig:non-auxetic eigen}.
$B(\xi)$ is elliptic for $\xi < 0$, degenerate at $\xi = 0$, hyperbolic for $\xi \lesssim 0.75$, degenerate again for $\xi \simeq 0.75$ and elliptic for $\xi \gtrsim 0.75$.

Also, note that for the pattern in Figure \ref{fig:non-auxetic sketch}, when $\xi = \frac\pi2$, $\mu_2(\xi) = 0$ and thus $\Gamma_{12} = +\infty$.
This type of issue happens typically when a pattern is fully open or fully closed, in which case the modeling performed in \cite{zheng2022continuum} does not apply any longer.
To prevent this type of issue, we let $\xi^- \in (-\frac\pi2,0)$, and $\xi^+ \in (0,\frac\pi2)$ and define the cut-off coefficients
\[ \widehat{\Gamma_{21}}(\xi) = \left\{ \begin{aligned} &\Gamma_{21}(\xi^+) \text{ if } \xi \ge \xi^+, \\
&\Gamma_{21}(\xi^-) \text{ if } \xi \le \xi^-, \\
&\Gamma_{21}(\xi) \text{ otherwise}, \end{aligned} \right.
\quad \widehat{\Gamma_{12}}(\xi) = \left\{ \begin{aligned} &\Gamma_{12}(\xi^+) \text{ if } \xi \ge \xi^+, \\
&\Gamma_{12}(\xi^-) \text{ if } \xi \le \xi^-, \\
&\Gamma_{12}(\xi) \text{ otherwise}, \end{aligned} \right. \]
and
\[ \hat{B}(\xi) = \begin{pmatrix}
-\widehat{\Gamma_{21}}(\xi) & 0 \\
0 & \widehat{\Gamma_{12}}(\xi) \\
\end{pmatrix}. \]
Therefore, $\hat{B}: L^2(\Omega) \to L^\infty(\Omega)$ and $\hat{B}$ is $K$-Lipschitz, where $K>0$, in the sense that $\hat{B} \in W^{1,\infty}(\mathbb R)$.
We choose $\xi^-$ and $\xi^+$ such that there exists $M > 0$, independent of $\xi$, 
\begin{equation}
\label{eq:bound on matrix}
\Vert \hat{B}(\xi)\Vert_{L^\infty(\Omega;\R)} \le M.
\end{equation}
Of course, $\hat{B}(\xi)$ is consistent with $B(\xi)$ only as long as $\xi^- \le \xi \le \xi^+$ a.e. in $\Omega$.
Let us now write a weak formulation of \eqref{eq:strong form}.
The bilinear form $a(\xi)$ is defined for $\zeta \in V$ and $\tilde{\zeta} \in V_0$ by
\begin{equation*}
a(\xi; \zeta, \tilde{\zeta}) := \int_\Omega \hat{B}(\xi) \nabla \zeta \cdot \overline{\nabla \tilde{\zeta}}.
\end{equation*}
Let $l$ be the linear form defined for $\tilde{\zeta} \in V_0$ by
\[ l(\tilde{\zeta}) := \int_{\partial \Omega_N} g \cdot \overline{\tilde{\zeta}}. \]

\subsection{Linear problem}
\label{sec:linear}
Let $\xi \in L^2(\Omega;\R)$ be fixed.
Because $\mathrm{det}(\hat{B}(\xi))$ can change sign, $a(\xi)$ is not coercive over $V_0 \times V_0$.
Therefore, one cannot apply the Lax--Milgram Lemma directly.
Instead, we follow \cite{chesnel2013t,ciarlet2022mathematical} and use a limiting absorption principle.
It consists in adding some dissipation to $a(\xi)$ through a purely imaginary term.
Let $\varepsilon > 0$ be a regularization parameter.
For $\zeta \in V$ and $\tilde{\zeta} \in V_0$, we define 
\[ a_\varepsilon(\xi; \zeta, \tilde{\zeta}) := \int_\Omega (\hat{B}(\xi) + i \varepsilon I) \nabla \zeta \cdot \overline{\nabla \tilde{\zeta}}, \]
where $i^2= -1$.

\begin{lemma}
\label{th:linear complex}
For all $\varepsilon > 0$, there exists a unique solution $\zeta_\varepsilon \in V_D$ to 
\begin{equation}
\label{eq:linear complex}
a_\varepsilon(\xi;\zeta_\varepsilon, \tilde{\zeta}) = l(\tilde{\zeta}), \quad \forall \tilde{\zeta} \in V_0.
\end{equation}
One has
\begin{equation}
\label{eq:bound}
\Vert \zeta_\varepsilon \Vert_V \le \left( \frac{M}\alpha + 1 \right) C_{tr}  \Vert \xi_D \Vert_{H^\frac12(\partial \Omega_D;\R)} + \frac{C_{tr}}{\alpha}\Vert g \Vert_{L^2(\partial \Omega_N;\R)},
\end{equation}
where $C_{tr} > 0$ is the constant from the trace theorem and $\alpha > 0$ is a coercivity constant described in the proof.
Note that both constants are independent of $\xi$.
\end{lemma}

\begin{proof}
As $\xi_D \in H^\frac12(\partial \Omega_D;\R)$, there exists $\zeta_D \in H^1(\Omega;\R)$, $\zeta_D = \xi_D$ on $\partial \Omega_D$ and one has
\[ \| \zeta_D \|_V \le C_{tr} \| \xi_D \|_{H^{\frac12}(\partial \Omega_D; \mathbb R)}, \]
where $C_{tr} > 0$ is the trace constant, see \cite[Theorem B.52]{ern_guermond}.
We define the linear form $L(\xi)$ for $\tilde{\zeta} \in V_0$ by
\begin{equation*}
L(\xi; \tilde{\zeta}) = l(\tilde{\zeta}) - \int_\Omega \hat{B}(\xi) \nabla \zeta_D \cdot \overline{\nabla \tilde{\zeta}}.
\end{equation*}
Let $\hat{\zeta}, \tilde{\zeta} \in V_0$.
One has
\[ a_\varepsilon(\xi; \hat{\zeta} + \zeta_D, \tilde{\zeta}) - l(\tilde{\zeta}) = a_\varepsilon(\xi; \hat{\zeta}, \tilde{\zeta}) - L(\xi; \tilde{\zeta}). \]
Therefore, we want to apply the Lax--Milgram lemma to find $\hat{\zeta}_\varepsilon \in V_0$ such that
\[ a_\varepsilon(\xi; \hat{\zeta}_\varepsilon, \tilde{\zeta}) = L(\xi;\tilde{\zeta}), \quad \forall \tilde{\zeta} \in V_0. \]
$a_\varepsilon(\xi)$ is sesquilinear and continuous over $V_0 \times V_0$.
Let $\zeta \in V_0$, one has
\[ \Im(a_\varepsilon(\xi;\zeta, \zeta)) \ge \varepsilon \Vert \nabla \zeta \Vert_{L^2(\Omega; \C)}^2. \]
One also has
\[ \Re(a_\varepsilon(\xi;\zeta, \zeta)) \ge -M \Vert \nabla \zeta \Vert_{L^2(\Omega; \C)}^2. \]
Let $C > \frac{M}{\varepsilon}$, one has
\[ \Re((1-iC)a_\varepsilon(\xi;\zeta,\zeta)) \ge ( C\varepsilon -M) \Vert \nabla \zeta \Vert_{L^2(\Omega; \C)}^2. \]
Let $C_\Omega > 0$ be the Poincar\'e constant \cite[Lemma~B.66]{ern_guermond} so that, for any $\zeta \in V_0$,
\[ \| \zeta \|_{L^2(\Omega)} \le C_\Omega \| \nabla \zeta \|_{L^2(\Omega)}.  \]
Letting $\gamma := \frac{1-iC}{\sqrt{1+C^2}}$, $|\gamma| = 1$, one has
\[ \Re(\gamma a_\varepsilon(\xi;\zeta,\zeta)) \ge \frac{C \varepsilon - M}{\sqrt{1+C^2}} \frac1{1+C_\Omega^2} \| \zeta \|^2_V =:  \alpha \| \zeta \|^2_V, \]
which we use as a definition for $\alpha$.
Therefore, $a_\varepsilon(\xi)$ is coercive over $V_0 \times V_0$.
As $L(\xi)$ is continuous over $V_0$, we apply the Lax--Milgram lemma \cite[Lemma~25.2]{ern2021finiteII} and get the existence of a unique $\hat{\zeta_\varepsilon} \in V_0$ such that $\zeta_\varepsilon := \hat{\zeta}_\varepsilon + \zeta_D$ solves \eqref{eq:linear complex}.
Note that this $\hat{\zeta_\varepsilon}$ depends on $\xi$ and $\varepsilon$.

Let us now prove the bound \eqref{eq:bound}.
As a consequence of the Lax--Milgram lemma, one has
\begin{align*}
\Vert \zeta_\varepsilon \Vert_V = \Vert \hat{\zeta}_\varepsilon + \zeta_D \Vert_V \le \|\zeta_D \|_V + \|\hat{\zeta}_\varepsilon \|_V & \le \|\zeta_D \|_V + \frac1\alpha \| L(\xi) \|_{V_0'}, \\
&  \le \left( 1 + \frac{M}\alpha \right) \| \zeta_D \Vert_V + \frac{C_{tr}}\alpha \|g \|_{L^2(\partial \Omega_N;\mathbb R)}, \\ 
& \le \left( 1 + \frac{M}\alpha \right) C_{tr} \Vert \xi_D \Vert_{H^\frac12(\partial \Omega_D;\R)} + \frac{C_{tr}}\alpha \|g \|_{L^2(\partial \Omega_N;\mathbb R)},
\end{align*}
where $\| \cdot \|_{V_0'}$ is the dual norm associated with $V_0$ endowed with the norm of $V$.
\end{proof}

\subsection{Nonlinear problem}
\label{sec:nonlinear}
Using the results from Section \ref{sec:linear}, we prove that there exists a solution to the approximate problem.

\begin{proposition}
For $\varepsilon > 0$, there exists $\xi_\varepsilon \in V_D$ solution of
\begin{equation}
\label{eq:fix point}
a_\varepsilon(\Re(\xi_\varepsilon); \xi_\varepsilon, \tilde{\zeta}) = l(\tilde{\zeta}), \quad \forall \tilde{\zeta} \in V_0.
\end{equation}
\end{proposition}

\begin{proof}
Let $\varepsilon > 0$ and $T_\varepsilon : L^2(\Omega; \C) \to V$ be the map such that for $\xi \in L^2(\Omega; \C)$, $T_\varepsilon \xi = \zeta_\varepsilon \in V_D$ is the unique solution of \[ a_\varepsilon(\Re(\xi);\zeta_\varepsilon, \tilde{\zeta}) = l(\tilde{\zeta}), \quad \forall \tilde{\zeta} \in V_0, \]
which exists due to Lemma \ref{th:linear complex}.
We will use the Schauder fixed point theorem \cite{gilbarg2015elliptic}.
The first step is to find an invariant domain for $T_\varepsilon$.
Let $B = \{\zeta \in V_D \ | \ \zeta \text{ verifies } \eqref{eq:bound} \}$, which is a compact convex set in $V$.
Using Lemma \ref{th:linear complex}, one has $T_\varepsilon B \subset B$.

Let us now show that $T_\varepsilon$ is continuous over $B$ for $\Vert \cdot \Vert_V$.
Let $\left(\xi_n\right)_{n \in \mathbb{N}}$ be a sequence of $B$ such that $\xi_n \mathop{\longrightarrow} \limits_{n \to +\infty} \xi \in B$ in $V$.
Let $\zeta = T_\varepsilon \xi$, and $\zeta_n = T_\varepsilon \xi_n$, for $n \in \mathbb{N}$.
We want to prove that $\zeta_n \mathop{\longrightarrow} \limits_{n \to +\infty} \zeta$ in $V$.
$\left( \zeta_n \right)_n$ is bounded in $V_D$, and thus there exists $\hat{\zeta} \in V_D$, up to a subsequence,  $\zeta_n \mathop{\rightharpoonup} \limits_{n \to +\infty} \hat{\zeta}$ weakly in $V$.
Using the classical Sobolev injection $V \subset L^2(\Omega;\C)$, $\zeta_n \to \hat{\zeta}$ strongly in $L^2(\Omega;\C)$, when $n \to +\infty$.
Let us show that $\hat{\zeta}$ is a solution of \eqref{eq:linear complex}.
Let $\varphi \in C^\infty(\Omega;\mathbb C)^2$ with $\varphi = 0$ on $\partial \Omega_D$.
By definition, one has
\[ \begin{aligned}
a_\varepsilon(\Re(\xi);\zeta_n, \varphi) - l(\varphi) & = a_\varepsilon(\Re(\xi);\zeta_n,\varphi) + a_\varepsilon(\Re(\xi_n);\zeta_n,\varphi), \\
& = \int_\Omega \left[ \hat{B}(\Re(\xi)) - \hat{B}(\Re(\xi_n)) \right] \nabla \zeta_n \cdot \nabla \varphi.
\end{aligned} \]
Therefore,
\[ \begin{aligned}
\left| a_\varepsilon(\Re(\xi);\zeta_n, \varphi) - l(\varphi) \right| & \le \left\|  \left[ \hat{B}(\Re(\xi)) - \hat{B}(\Re(\xi_n)) \right] \nabla \zeta_n \right\|_{L^1(\Omega)} \| \nabla \varphi \|_{L^\infty(\Omega)}, \\
& \le K \| \xi - \xi_n \|_{L^2(\Omega)} \|\zeta_n \|_V \| \nabla \varphi \|_{L^\infty(\Omega)}, \\
& \le C K\| \xi - \xi_n \|_{L^2(\Omega)} \| \nabla \varphi \|_{L^\infty(\Omega)} \mathop{\longrightarrow}_{n \to +\infty} 0,
\end{aligned} \]
since $\hat{B}$ is $K$-Lipschitz, $\|\zeta_n \|_V \le C$ and $\xi_n \to \xi$ strongly in $V$.
Note that since $\zeta_n \rightharpoonup \hat{\zeta}$ weakly in $V$, one has 
\[ a_\varepsilon(\Re(\xi); \hat{\zeta}, \varphi) = l(\varphi). \]
We finish proving that $\hat{\zeta} \in V_D$ is a solution of \eqref{eq:linear complex} by using a density argument.
Let $\tilde{\zeta} \in V_0$.
There exists a sequence $(\varphi_n)_{n \in \mathbb N} \in \mathcal C^\infty(\Omega; \mathbb C)$, with compact support in $\partial \Omega_D$, such that $\| \tilde{\zeta} - \varphi_n \Vert_V \to 0$, as $n \to \infty$.
One  has
\begin{equation}
\label{eq:intermediate fix point}
a_\varepsilon(\Re(\xi); \hat{\zeta}, \tilde{\zeta}) - l(\tilde{\zeta})
= a_\varepsilon(\Re(\xi); \hat{\zeta}, \varphi_n) - l(\varphi_n) + a_\varepsilon(\Re(\xi); \hat{\zeta}, \tilde{\zeta} - \varphi_n) - l(\tilde{\zeta} - \varphi_n ).
\end{equation}
The sum of the first two terms in the right-hand side of \eqref{eq:intermediate fix point} is zero since $\varphi_n \in \mathcal C^\infty(\Omega; \mathbb C)$ with $\varphi_n = 0$ on $\partial \Omega_D$.
The last two terms in the right-hand side of \eqref{eq:intermediate fix point} converge to zero since $\zeta_n \to \tilde{\zeta}$ strongly in $V$.
Therefore, $\hat{\zeta} \in V_D$ solves \eqref{eq:linear complex}, which has for unique solution $\zeta \in V_D$, and thus $\hat{\zeta} = \zeta$.

Let us show that $\zeta_n \to \zeta$ strongly in $V$ when $n \to +\infty$.
One has
\begin{align*}
i \varepsilon \Vert \nabla \zeta_n - \nabla \zeta \Vert_{L^2(\Omega; \C)}^2 + \int_\Omega \hat{B}(\Re(\xi_n)) \nabla(\zeta_n - \zeta) \cdot \overline{\nabla(\zeta_n - \zeta)} & = a_\varepsilon(\Re(\xi_n); \zeta_n - \zeta, \zeta_n - \zeta), \\
& =a_\varepsilon(\Re(\xi_n); \zeta_n, \zeta_n - \zeta) - a_\varepsilon(\Re(\xi_n); \zeta, \zeta_n - \zeta), \\
& = l(\zeta_n - \zeta) - a_\varepsilon(\Re(\xi_n); \zeta, \zeta_n - \zeta), \\
& \mathop{\longrightarrow}_{n \to \infty} 0,
\end{align*}
as $\xi_n \to \xi$ strongly in $L^2(\Omega;\C)$, $\zeta_n \rightharpoonup \zeta$ weakly in $V$, and $\zeta_n \to \zeta$ strongly in $L^2(\Omega;\C)$.
Taking the imaginary part, one has $\zeta_n \to \zeta$ strongly in $V$.
Also, as the solution of \eqref{eq:linear complex} is unique, the full sequence $\left( \zeta_n \right)_n$ converges towards $\zeta$ in V.
We conlude with the Schauder fixed point theorem.
\end{proof}

\begin{remark}
Typical techniques for proving the uniqueness of the solution for small enough boundary condition as in \cite[Section~8.7]{brenner} do not apply here since they would require $\xi \in W^{1,p}(\Omega)$, with $p > 2$.
We cannot expect such regularity given the fact that we only have $\hat{B}(\xi) \in L^\infty(\Omega)$.
\end{remark}

\section{Approximate Discrete problem}
\label{sec:discrete}
This section is dedicated to approximating the solutions of \eqref{eq:fix point}.

\subsection{Discrete setting}
Let $\left(\mathcal{T}_h \right)_h$ be a family of quasi-uniform and shape regular triangulations \cite{ern_guermond}, perfectly fitting $\Omega$.
For a cell $c \in \mathcal{T}_h$, let $h_c := \mathrm{diam}(c)$ be the diameter of c.
Then, we define $h := \max_{c \in \mathcal{T}_h} h_c$ as the mesh parameter for a given triangulation $\mathcal{T}_h$.
Let $V_h := \mathbb{P}^1(\mathcal{T}_h;\C)$, the space of affine Lagrange polynomials with complex values.
As the solutions we try to compute are only $H^1$, we use the Scott--Zhang interpolator written as $\mathcal{I}_h$, see \cite{ern_guermond} for details.
As $\xi_D \in H^\frac12(\partial \Omega_D;\R)$, there exists $\zeta_D \in H^1(\Omega;\R)$, $\zeta_D = \xi_D$ on $\partial \Omega_D$.
We define the following solution set as
\[ V_{hD} := \{\zeta_h \in V_{h}; \Re(\zeta_h) =  \mathcal{I}_h \zeta_D \text{ and } \Im(\zeta_h) = 0 \text{ on } \partial \Omega_D  \}, \]
and its associated homogeneous space as
\[ V_{h0} := \{\zeta_h \in V_{h}; \zeta_h =  0 \text{ on } \partial \Omega_D  \}. \]

\subsection{Discrete problem}
\label{sec:discrete pb}
Let $\varepsilon > 0$. 
We focus on the discrete nonlinear problem which consists in searching for $\xi_h \in V_{hD}$,
\begin{equation}
\label{eq:discrete nonlinear}
a_\varepsilon(\Re(\xi_h); \xi_h, \tilde{\zeta}_h) = l(\tilde{\zeta}_h), \quad \forall \tilde{\zeta}_h \in V_{h0}.
\end{equation}

\begin{proposition}
\label{th:discrete nonlinear}
There exists a solution $\xi_h \in V_{hD}$ to \eqref{eq:discrete nonlinear}.
\end{proposition}

\begin{proof}
Let $\xi_h \in V_{hD}$.
Let us first focus on the following linear problem: search for $\zeta_h \in V_{hD}$, such that
\begin{equation}
\label{eq:discrete linear}
a_\varepsilon(\Re(\xi_h); \zeta_h, \tilde{\zeta}_h) = l(\tilde{\zeta}_h), \quad \forall \tilde{\zeta}_h \in V_{h0}.
\end{equation}
Using the Lax--Milgram lemma (similarly to the proof of Lemma \ref{th:linear complex}), we can show that there exists a unique solution $\zeta_h \in V_{hD}$ to \eqref{eq:discrete linear}, and one has
\begin{equation}
\label{eq:discrete bound}
\Vert \zeta_h \Vert_V \le C' C_{tr}\left(1 + \frac{M}\alpha \right) \Vert \xi_D \Vert_{H^\frac12(\partial \Omega_D;\R)} + \frac{C' C_{tr}}\alpha \Vert g \Vert_{L^2(\partial \Omega_N;\R)},
\end{equation}
where $C_{tr} > 0$ is the constant from the trace theorem, $\alpha$ is the coercivity constant, and $C' > 0$ is the interpolation constant, see \cite[Lemma~1.130]{ern_guermond}.

We define the fixed point map $T_h : V_{hD} \to V_{hD}$, such that $T_h \xi_h = \zeta_h$, where $\zeta_h$ is the solution of \eqref{eq:discrete linear}.
Let 
\[ B = \left\{ \xi_h \in V_{hD} \ | \ \xi_h \text{ verifies } \eqref{eq:discrete bound} \right\}. \]
Using \eqref{eq:discrete bound}, one has $T_h B \subset B$.
Let us show that $T_h$ is continuous over $B$.
Let $\left(\xi_n\right)_{n \in \mathbb{N}}$ be a sequence of $B$ such that $\xi_n \mathop{\longrightarrow} \limits_{n \to +\infty} \xi \in B$ strongly in $V$.
Let $\zeta = T_h \xi$, and $\zeta_n = T_h \xi_n$, for $n \in \mathbb{N}$.
We want to prove that $\zeta_n \mathop{\longrightarrow} \limits_{n \to +\infty} \zeta$ strongly in $V$.

$\left( \zeta_n \right)_n$ is bounded in $B$, and thus there exists $\hat{\zeta} \in V_{hD}$, up to a subsequence,  $\zeta_n \mathop{\rightharpoonup} \limits_{n \to +\infty} \hat{\zeta}$ weakly in $V$.
By Sobolev injection, $\zeta_n \to \hat{\zeta}$ strongly in $L^2(\Omega;\C)$, when $n \to +\infty$.
Let us show that $\hat{\zeta} = \zeta$ by showing that $\hat{\zeta}$ solves
\[ a_\varepsilon(\Re(\xi);\hat{\zeta}, \tilde{\zeta_h}) = l(\tilde{\zeta_h}), \quad \forall \tilde{\zeta}_h \in V_{h0}. \] 
Let $\tilde{\zeta}_h \in V_{h0}$.
One has
\begin{equation}
\begin{aligned}
\left| a_\varepsilon(\Re(\xi); \zeta_n, \tilde{\zeta}_h) - l(\tilde{\zeta}_h) \right| & = \left| a_\varepsilon(\Re(\xi); \zeta_n, \tilde{\zeta}_h) - a_\varepsilon(\Re(\xi_n); \zeta_n, \tilde{\zeta}_h) \right|, \\
& = \left|\int_\Omega \left[ \hat{B}(\Re(\xi)) - \hat{B}(\Re(\xi_n)) \right] \nabla \zeta_n \cdot \tilde{\zeta}_h \right|, \\
& \le K \| (\xi - \xi_n) \nabla \zeta_n \|_{L^1(\Omega;\mathbb C)} \| \nabla \tilde{\zeta}_h \|_{L^\infty(\Omega;\mathbb C)}, \\
& \le K \| \xi - \xi_n \|_{L^2(\Omega)} C \| \nabla \tilde{\zeta}_h \|_{L^\infty(\Omega;\mathbb C)} \mathop{\longrightarrow}_{n \to +\infty} 0,
\end{aligned}
\end{equation}
since $\hat{B}$ is $K$-Lipschitz and $(\zeta_n)_n$ is bounded in $H^1(\Omega)$.
Note that since $\zeta_n \rightharpoonup \hat{\zeta}$ weakly in $V$, one has
\[ a_\varepsilon(\Re(\xi); \zeta_n, \tilde{\zeta}_h) \mathop{\longrightarrow}_{n \to +\infty} a_\varepsilon(\Re(\xi); \hat{\zeta}, \tilde{\zeta}_h), \]
and thus $\hat{\zeta} = T_h \xi = \zeta$.

Let us show that $\zeta_n \to \zeta$ strongly in $V$ when $n \to +\infty$.
One has
\begin{align*}
&i \varepsilon \Vert \nabla \zeta_n - \nabla \zeta \Vert_{L^2(\Omega; \C)}^2 + \int_\Omega \hat{B}(\Re(\xi_n)) \nabla(\zeta_n - \zeta) \cdot \overline{\nabla(\zeta_n - \zeta)} \\
= & \, a_\varepsilon(\Re(\xi_n); \zeta_n - \zeta, \zeta_n - \zeta), \\
= & \, a_\varepsilon(\Re(\xi_n); \zeta_n, \zeta_n - \zeta) - a_\varepsilon(\Re(\xi_n); \zeta, \zeta_n - \zeta), \\
= & \, l(\zeta_n - \zeta) - a_\varepsilon(\Re(\xi_n); \zeta, \zeta_n - \zeta) \mathop{\longrightarrow}_{n \to \infty} 0, \\
\end{align*}
as $\| \xi_n \|_V \le C$, $\zeta_n \rightharpoonup \zeta$ weakly in $V$, and $\zeta_n \to \zeta$ strongly in $L^2(\Omega;\C)$.
By taking the imaginary part, one has $\zeta_n \to \zeta$ strongly in $V$.
Also, as the solution of \eqref{eq:discrete linear} is unique, the full sequence $\left( \zeta_n \right)_n$ converges towards $\zeta$ in V.
Thus $T_h$ is continuous and the Brouwer fixed point \cite{brezis} gives the desired result.

\end{proof}

\subsection{Convergence}
We study the convergence of the method when $h \to 0$, for $\varepsilon > 0$ fixed.

\begin{theorem}
Let $\left( \xi_h \right)_{h > 0}$ be a sequence of solutions of \eqref{eq:discrete nonlinear}.
There exists $\xi_\varepsilon \in V_D$, solution of \eqref{eq:fix point}, such that, up to a subsequence,
\[ \xi_h \mathop{\longrightarrow}_{h \to 0} \xi_\varepsilon \text{ strongly in } V. \]
\end{theorem}

\begin{proof}
As $\left( \xi_h \right)_{h > 0}$ is bounded in $V_D$, according to Proposition \ref{th:discrete nonlinear}, there exists $\xi_\varepsilon \in V_D$, up to a subsequence,  $\xi_h\rightharpoonup \xi_\varepsilon$, weakly in $V$, when $h \to 0$.
Using a compact Sobolev injection, one has $\xi_h \to \xi_\varepsilon$, strongly in $L^2(\Omega;\C)$, when $h \to 0$.
We now want to show that $\xi_\varepsilon \in V_D$ is a solution of \eqref{eq:fix point}.

Let $\varphi \in C^\infty(\Omega;\mathbb C)$ with $\varphi = 0$ on $\partial \Omega_D$
For this entire proof, we use the interpolation operator, still written as $\mathcal I_h$ from \cite[Section~22.4]{ern2021finiteI}.
Note that due to the regularity of $\varphi$, one has $\nabla \mathcal I_h \varphi \to \nabla \varphi$ strongly in $L^2(\Omega;\mathbb C)$, as $h \to 0$, see \cite[Lemma~1.130]{ern_guermond}.
Therefore, when $h \to 0$, one has $l(\mathcal{I}_h \varphi) \to l(\varphi)$.
One also has
\begin{equation}
\label{eq:intermediate}
\begin{aligned}
& a_\varepsilon(\Re(\xi_\varepsilon);\xi_h,\mathcal I_h \varphi) - l(\mathcal I_h \varphi) \\
=& \, a_\varepsilon(\Re(\xi_\varepsilon);\xi_h,\mathcal I_h \varphi) - a_\varepsilon(\Re(\xi_h);\xi_h,\mathcal I_h \varphi) , \\
=& \int_\Omega [\hat{B}(\Re(\xi_\varepsilon)) - \hat{B}(\Re(\xi_h)) ] \nabla \xi_h \cdot \nabla \mathcal I_h \varphi. \\
\end{aligned}
\end{equation}
Therefore,
\begin{align*}
\left| a_\varepsilon(\Re(\xi_\varepsilon);\xi_h,\mathcal I_h \varphi) - l(\mathcal I_h \varphi) \right| 
& \le \left\Vert [\hat{B}(\Re(\xi_\varepsilon)) - \hat{B}(\Re(\xi_h))] \nabla \xi_h \right\Vert_{L^1(\Omega; \mathbb C)} \|\nabla \mathcal I_h \varphi \|_{L^\infty(\Omega; \mathbb C)}, \\
& \le C K \| \xi_h - \xi_\varepsilon \|_{L^2(\Omega)} C' \| \varphi \|_{W^{1,\infty}(\Omega)} \mathop{\longrightarrow}_{h \to 0} 0, 
\end{align*}
since $\hat{B}$ is $K$-Lipschitz, $\| \xi_h \|_V \le C$, $\xi_h \to \xi_\varepsilon$ strongly in $L^2$ and $C' > 0$ is the stability constant from \cite[Corollary~22.15]{ern2021finiteI}.
Since $\xi_h \rightharpoonup \xi_\varepsilon$ weakly in $V$, for all $\varphi \in C^\infty(\Omega; \mathbb C)$ with $\varphi = 0$ on $\partial \Omega_D$, one has
\begin{equation}
\label{eq:intermediate 4}
a_\varepsilon(\Re(\xi_\varepsilon); \xi_\varepsilon, \varphi) = l(\varphi). 
\end{equation}
We finish proving that $\xi_\varepsilon \in V_D$ is a solution of \eqref{eq:fix point} by using a density argument.
Let $\tilde{\zeta} \in V_0$.
There exists a sequence $(\varphi_n)_{n \in \mathbb N} \in \mathcal C^\infty(\Omega; \mathbb C)$, with $\varphi_n = 0$ on $\partial \Omega_D$, such that $\| \tilde{\zeta} - \varphi_n \Vert_V \to 0$, as $n \to \infty$.
One thus has
\begin{equation}
\label{eq:intermediate 5}
a_\varepsilon(\Re(\xi_\varepsilon); \xi_\varepsilon, \tilde{\zeta}) - l(\tilde{\zeta}) \\
= a_\varepsilon(\Re(\xi_\varepsilon); \xi_\varepsilon, \tilde{\zeta} - \varphi_n)
+ a_\varepsilon(\Re(\xi_\varepsilon); \xi_\varepsilon, \varphi_n) - l(\varphi_n) + l(\varphi_n - \tilde{\zeta}).
\end{equation}
Note that, using \eqref{eq:intermediate 4}, the sum of the middle two terms in \eqref{eq:intermediate 5} is exactly zero.
The last term in the right-hand side of \eqref{eq:intermediate 5} converges to zero since $\| \tilde{\zeta} - \varphi_n \Vert_V \to 0$, as $n \to \infty$.
Regarding the first term in the right-hand side of \eqref{eq:intermediate 5} , one has
\[ |a_\varepsilon(\Re(\xi_\varepsilon); \xi_\varepsilon, \tilde{\zeta} - \varphi_n)| \le (M + \varepsilon) \| \xi_\varepsilon \|_{H^1(\Omega)} \| \tilde{\zeta} - \varphi_n \Vert_V \to 0, \] 
as $n \to \infty$.
And therefore, $\xi_\varepsilon \in V_D$ is a solution of \eqref{eq:fix point}.

Let us now show the strong convergence of the gradients.
One has
\begin{align*}
&i \varepsilon \| \nabla \xi_h - \nabla \xi_\varepsilon \|^2_{L^2(\Omega;\mathbb C)} + \int_\Omega \hat{B}(\Re(\xi_h)) \ \nabla (\xi_h - \xi_\varepsilon) \nabla (\xi_h - \xi_\varepsilon) \\
= & \, a_\varepsilon(\Re(\xi_h);\xi_h - \xi_\varepsilon, \xi_h - \xi_\varepsilon) = a_\varepsilon(\Re(\xi_h);\xi_h, \xi_h - \xi_\varepsilon) - a_\varepsilon(\Re(\xi_h);\xi_\varepsilon, \xi_h - \xi_\varepsilon), \\
= & \, l(\xi_h - \xi_\varepsilon)  - a_\varepsilon(\Re(\xi_h);\xi_\varepsilon, \xi_h - \xi_\varepsilon).
\end{align*}
The first term converges to zero since $\xi_h \to \xi_\varepsilon$ strongly in $L^2(\Omega; \mathbb C)$.
The second term also converges to zero since $\xi_h$ is bounded in $V$ and $\nabla \xi_h \rightharpoonup \nabla \xi_\varepsilon$ weakly in $L^2(\Omega; \mathbb C)$.
Taking the imaginary part of the expression above, one gets that $\xi_h \to \xi_\varepsilon$ strongly in $V$.
\end{proof}

\section{Numerical tests}
\label{sec:test}
The numerical experiments consist in solving \eqref{eq:discrete nonlinear} for various boundary conditions and different material parameters.
Equation \eqref{eq:discrete nonlinear} is solved through Newton iterations.
Note that within the range of values of $\xi$ obtained in these numerical experiments, one has $\hat{B}(\xi) = B(\xi)$.
Therefore, no cut-off of the eigenvalues of $B(\xi)$ is necessary for the numerical experiments presented below.
The implementation is performed in\texttt{firedrake}, using its configuration for complex numbers, see \cite{firedrake}.
Then, $\gamma_h$, and $y_{\mathrm{eff},h}$ are computed from $\Re(\xi_h)$ through least-squares and Equations \eqref{eq:gamma} and \eqref{eq:yeff}.
These results are then compared to experimental results.
The regularization parameter $\varepsilon$ is chosen to be as small as possible while still providing a good agreement with the experimental results.
The experimental data are recovered using the image recognition capabilities of \texttt{matlab},
see \cite[Supplementary Material~7]{zheng2022continuum}.

\subsection{Auxetic kirigami}
\label{sec:elliptic}
The pattern is sketched in Figure \ref{fig:auxetic sketch}.
For this pattern, one has $\alpha=-0.9$ and $\beta=0.9$.
In \cite{czajkowski2022conformal,zheng2022continuum}, a solution method using quasi-conformal maps was proposed to compute patterns for which $\alpha=-\beta$.
The domain is $\Omega = (0,L) \times (0,L)$, where $L=15$.
Homogeneous Neumann boundary conditions are imposed on the top and bottom surfaces and the nonhomogeneous Dirichlet condition $\xi_D$ is imposed on the left and right surfaces.

Figure \ref{fig:auxetic results num} shows the post-processing of an experiment that we will try to replicate numerically.
We notice that, for this experiment, $\xi \in [0, \frac\pi3]$.
In that case, one has $-\Gamma_{21}(\xi) > 0$ and thus \eqref{eq:strong form} is strictly elliptic.
The regularization parameter is thus chosen as $\varepsilon = 0$, as this problem remains elliptic.

We consider a mesh of size $h=0.005$, with $119,817$ dofs.
Figure \ref{fig:auxetic results exp} shows the experimental results.
\begin{figure}[!htp]
\centering
\begin{subfigure}[b]{0.45\textwidth}
        \centering
        \includegraphics[scale=.15]{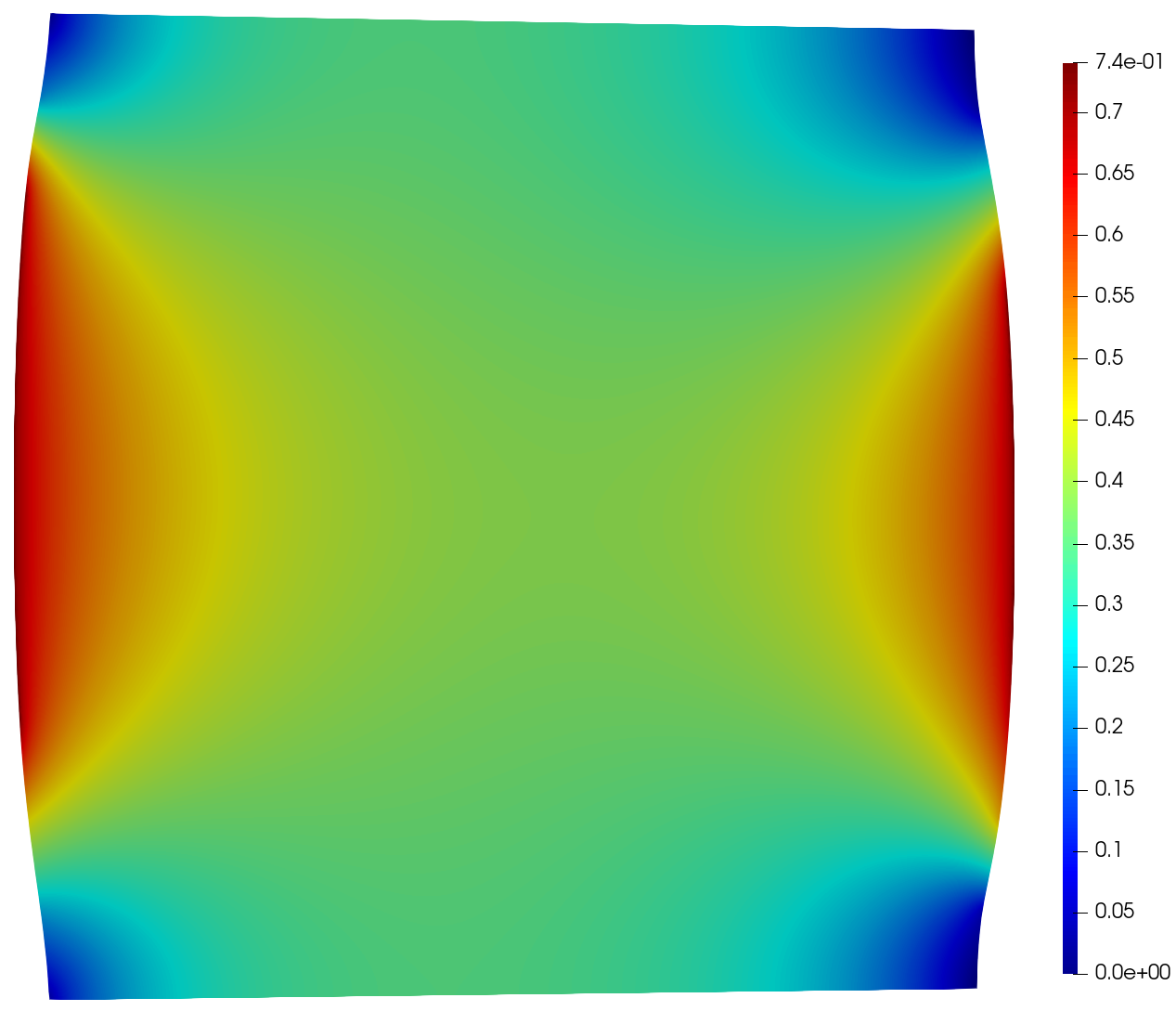}
        \caption{$\Re(\xi_h)$}
        \label{fig:auxetic results num}
    \end{subfigure}
    \hfill
    \begin{subfigure}[b]{0.45\textwidth}
        \centering
        \includegraphics[scale=.65]{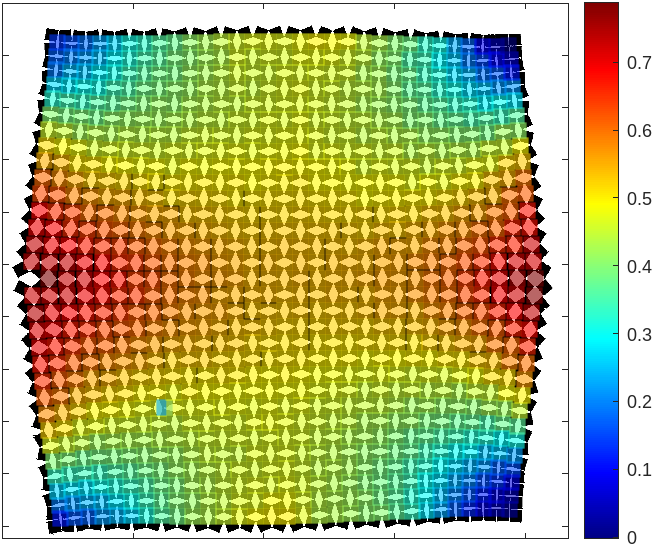}
        \caption{Experimental result}
        \label{fig:auxetic results exp}
    \end{subfigure}
    \caption{Auxetic kirigami}
\end{figure}
The numerical results are very similar to the ones obtained in \cite{zheng2022continuum}.

\subsection{Non-auxetic kirigami}
\label{sec:hyperbolic}
The pattern is sketched in Figure \ref{fig:non-auxetic sketch}.
For this pattern, one has $\alpha=-0.9$ and $\beta=0$.
In \cite{zheng2022continuum}, a solution method using traveling waves was proposed to compute patterns for which $\beta = 0$.
For $\xi \in [0, \frac\pi3]$, $-\Gamma_{21}(\xi) \le 0$ and thus \eqref{eq:strong form} is degenerate hyperbolic.
The domain is $\Omega = (0,L) \times (0,L)$, where $L=15$.
Homogeneous Neumann boundary conditions are imposed on the top and bottom surfaces and a nonhomogeneous Dirichlet condition $\xi_D$ is imposed on the left and right surfaces.

We consider a mesh of size $h = 4.7 \cdot 10^{-2}$, with $205,441$ dofs.
Figure \ref{fig:non-auxetic results} shows $\Re(\xi_h)$ in the deformed configuration for various values of $\varepsilon$.
\begin{figure}[!htp]
\centering
	\begin{subfigure}[b]{0.45\textwidth}
        \centering
        \includegraphics[scale=.5]{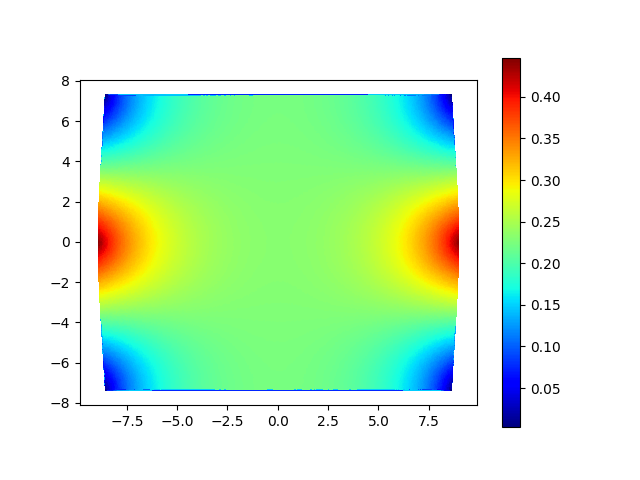}
        \caption{$\varepsilon = 1$}
    \end{subfigure}
    \hfill
	\begin{subfigure}[b]{0.45\textwidth}
        \centering
        \includegraphics[scale=.5]{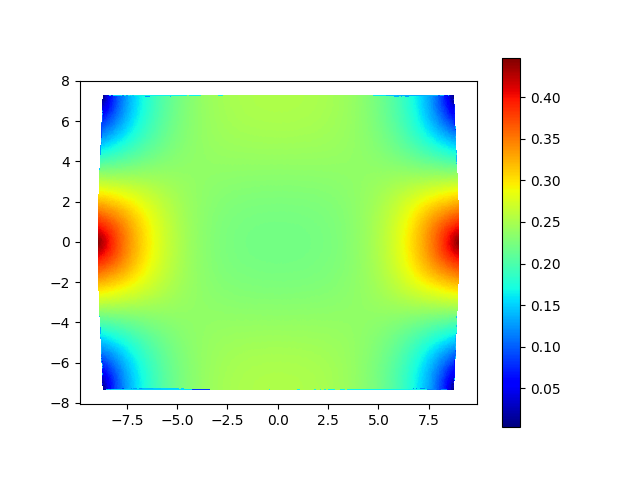}
        \caption{$\varepsilon = 0.4$}
    \end{subfigure}
    
    \begin{subfigure}[b]{0.45\textwidth}
        \centering
        \includegraphics[scale=.5]{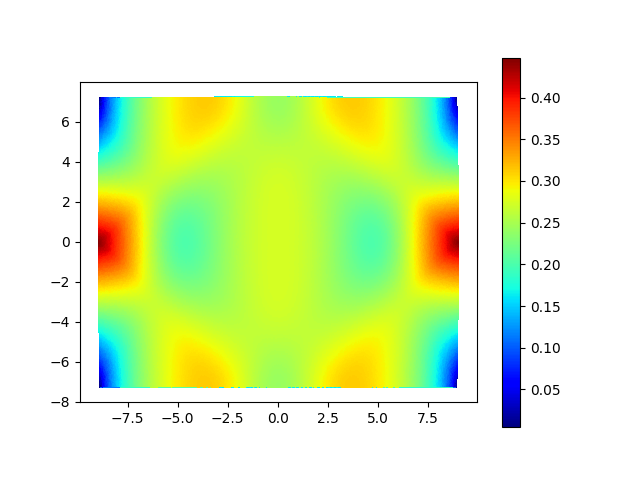}
        \caption{$\varepsilon = 0.1$}
    \end{subfigure}
    \hfill
    \begin{subfigure}[b]{0.45\textwidth}
        \centering
        \includegraphics[scale=.5]{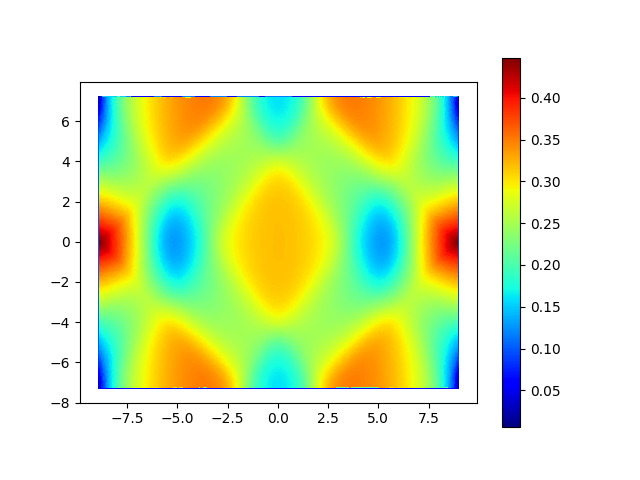}
        \caption{$\varepsilon = 0.05$}
    \end{subfigure}
    \caption{Non-auxetic kirigami: $\Re(\xi_h)$.}
\label{fig:non-auxetic results}
\end{figure}
For $\varepsilon \ge 1$, we start seeing results characteristic of an elliptic equation, which is not the behavior observed in the experiments.
The computation diverges with $\varepsilon = 0.01$.
Figure \ref{fig:hyperbolic results exp} shows the experimental result.
\begin{figure}[!htp]
\centering
\includegraphics[scale=.7]{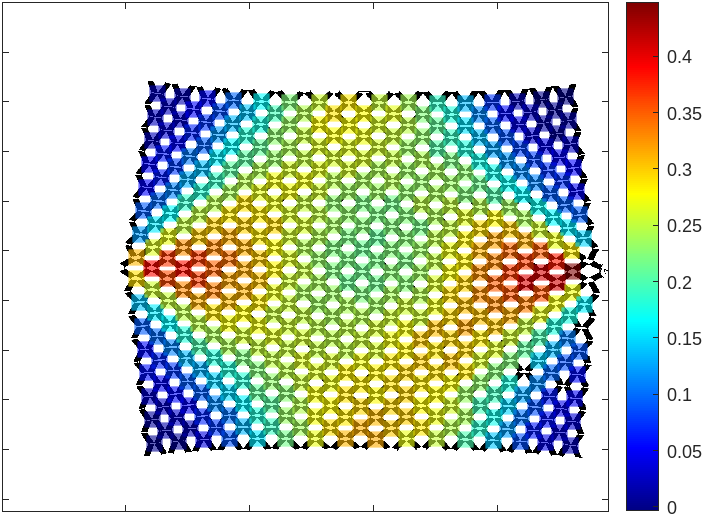}
\caption{Non-auxetic kirigami: experimental result.}
\label{fig:hyperbolic results exp}
\end{figure}
In Figure \ref{fig:hyperbolic results exp}, one can see a slight depression in the middle of the sample.
In Figure \ref{fig:non-auxetic results}, one can observe a similar depression for $\varepsilon = 0.4$.
With respect to the results of \cite{zheng2022continuum}, this numerical result shows a better agreement with the experimental data.

\subsection{Mixed type kirigami}
\label{sec:mixed}
The pattern is sketched in Figure \ref{fig:mixed sketch}.
\begin{figure}[!htp]
\centering
\includegraphics[scale=.5]{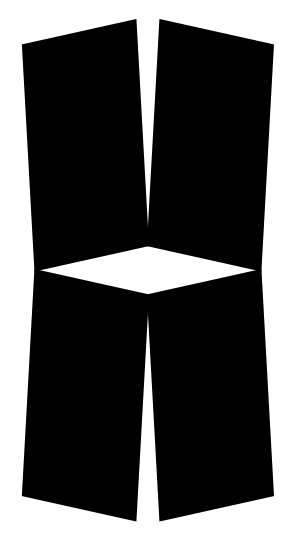}
\caption{Mixed type kirigami: sketch from Paolo Celli.}
\label{fig:mixed sketch}
\end{figure}
For this pattern, one has $\alpha=-1.6$ and $\beta=0.4$.
The domain is $\Omega = (0,L) \times (0,L)$, where $L=15$.
Homogeneous Neumann boundary conditions are imposed on the top and bottom surfaces and a nonhomogeneous Dirichlet condition $\xi_D$ is imposed on the left and right surfaces.
A main contribution of this paper is to be able to approximate solutions for this pattern.
Indeed, \cite{zheng2022continuum} proposed solutions methods only when $\alpha=-\beta$ or $\beta=0$, which is not the case for this pattern.

We consider a mesh of size $h = 4.7 \cdot 10^{-2}$, with $205,441$ dofs.
Figure \ref{fig:mixed results exp} shows the experimental results.
\begin{figure}[!htp]
\centering
\includegraphics[scale=.5]{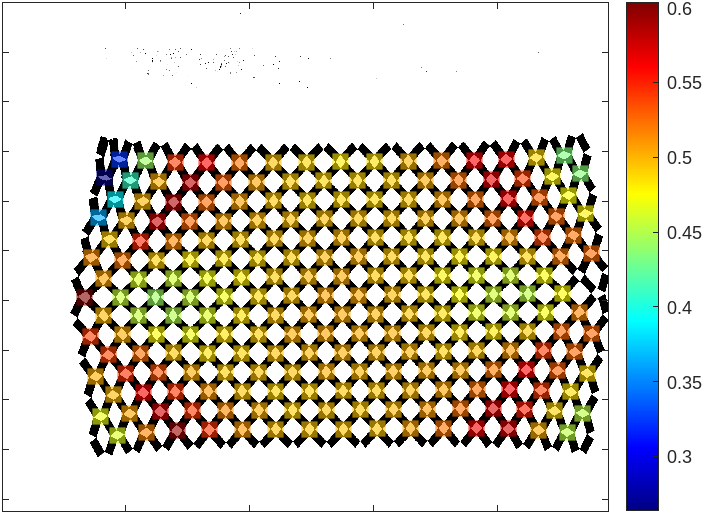}
\caption{Mixed type kirigami: experimental result.}
\label{fig:mixed results exp}
\end{figure}
Figure \ref{fig:mixed results} shows $\Re(\xi_h)$ in the deformed configuration for various values of $\varepsilon$.
\begin{figure}[!htp]
\centering
	\begin{subfigure}[b]{0.45\textwidth}
        \centering
        \includegraphics[scale=.5]{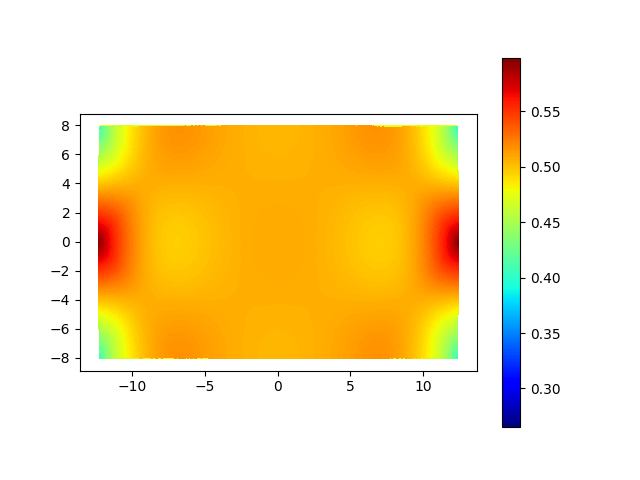}
        \caption{$\varepsilon = 0.15$}
    \end{subfigure}
    \hfill
	\begin{subfigure}[b]{0.45\textwidth}
        \centering
        \includegraphics[scale=.5]{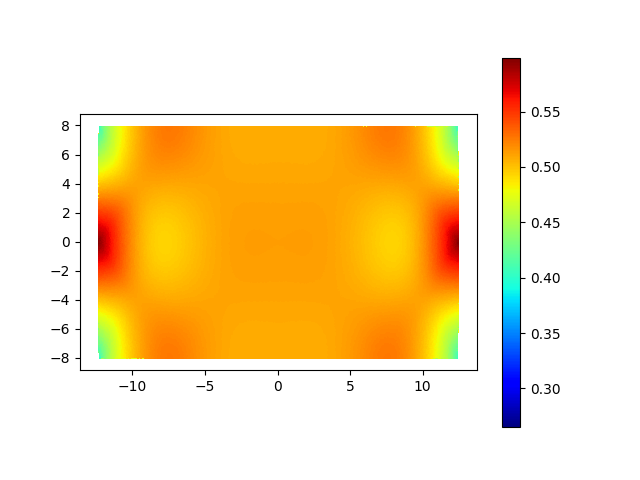}
        \caption{$\varepsilon = 0.1$}
    \end{subfigure}
    
    \begin{subfigure}[b]{0.45\textwidth}
        \centering
        \includegraphics[scale=.5]{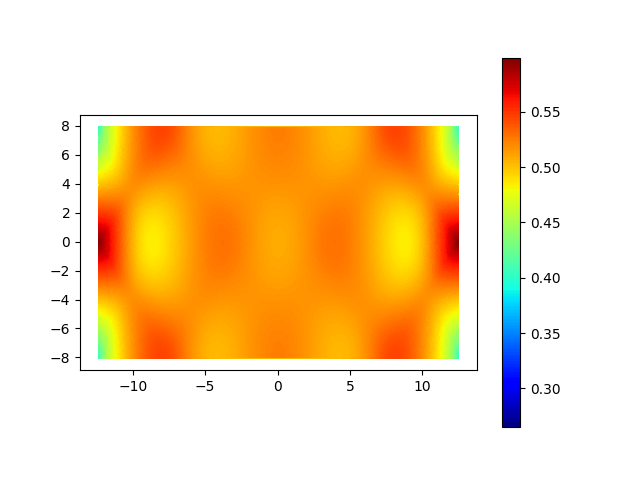}
        \caption{$\varepsilon = 0.05$}
    \end{subfigure}
    \hfill
    \begin{subfigure}[b]{0.45\textwidth}
        \centering
        \includegraphics[scale=.5]{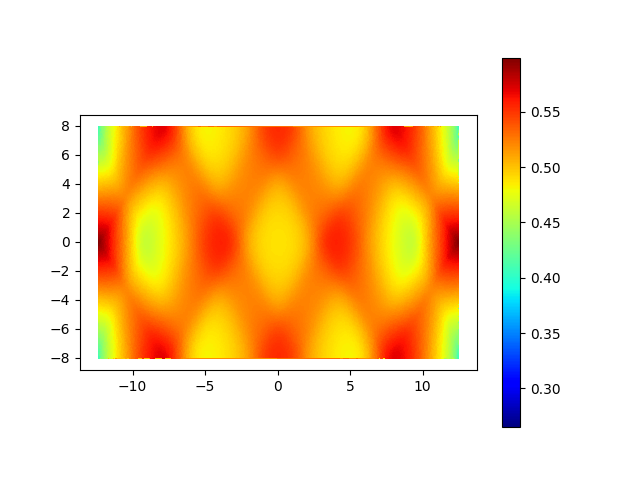}
        \caption{$\varepsilon = 0.02$}
    \end{subfigure}
    \caption{Mixed type kirigami: $\Re(\xi_h)$.}
\label{fig:mixed results}
\end{figure}
We notice that for $\varepsilon=0.15$ and $\varepsilon = 0.1$, the numerical results are in accordance with the experimental results as the two depressions, on the left and right, are well represented.
For $\varepsilon \le 0.05$, we notice that the computation seems to present too many depressions.
For a value $\varepsilon = 10^{-2}$, the Newton solver diverges.
This numerical test shows a good agreement with the experimental results.

\subsection{Discussion}
It has been noted in Sections \ref{sec:hyperbolic} and \ref{sec:mixed} that, for small enough $\varepsilon$, the agreement between the experimental and numerical results was not satisfactory.
Two phenomena seem to be at play here.
The first one is that since the hinges are not perfect in the experimental setting, there is some energy that is dissipated in their deformation.
This seems to justify that a small but non-zero value of $\varepsilon$ is physically relevant to reproduce the experimental results.

The second phenomenon at play is related to the behavior of viscosity solutions \cite{crandall1992user}.
Note that the notion of viscosity solution applies to nonlinear elliptic equations and not nonlinear hyperbolic equations.
Even though a parallel can be drawn between the methodology proposed in this paper and the notion of viscosity solution, it is in no case a proof.

Viscosity solutions, when they are defined as limits of regularized solutions, can sometimes ``lose" a part of their Dirichlet boundary conditions.
Let us consider the following example.
Let $\Omega = (0,1)$.
We want so solve for $x \in \Omega$,
\[ u'(x) = 1 \text{ such that } u(0) = u(1) = 0. \]
Let $\varepsilon > 0$.
We consider the regularized equation, for $x \in \Omega$,
\[ u_\varepsilon'(x) - \varepsilon u_\varepsilon''(x) = 1 \text{ such that } u_\varepsilon(0) = u_\varepsilon(1) = 0. \]
One can show that this equation has a unique solution.
Upon computing it, one can show that for all $\varepsilon > 0$, $u_\varepsilon(x) = 0$ and for $x \in \Omega$, $u_\varepsilon(x) \to x $, as $\varepsilon \to 0$.
In this case, $u(x) := x$ is the viscosity limit of $u_\varepsilon$ and is a viscosity solution of the original problem but it is not a classical solution since $u(1) \neq 0$.
The boundary condition on the right of the domain has been ``lost".

A similar situation is observed for the regularization used in this paper.
Let $\Omega \subset \mathbb{R}^2$ be the domain sketched in Figure \ref{fig:domain}.
\begin{figure} [!htp]
   \centering
   \begin{tikzpicture}
   \draw (0,1) -- (1,0) -- (0,-1) -- (-1, 0) -- cycle; 
   \draw node at (1,-0.5) {$(1,0)$};
   \draw node at (-1.3,0.3) {$(-1,0)$};
   \draw node at (-.7,-1) {$(0,-1)$};
   \draw node at (0.5,1) {$(0,1)$};
   
   \draw[->] (0,-1.5) -- (0,1.5);
   \draw[below] node at (1.5,0) {$x$};
   \draw[->] (-1.5,0) -- (1.5,0);
   \draw[right] node at (0,1.5) {$y$};

   \end{tikzpicture}
   \caption{Discussion: geometry of the domain $\Omega$.}
   \label{fig:domain}
\end{figure} 
Let us consider the problem of searching for $u:\Omega \to \mathbb{R}$, such that 
\begin{equation}
\label{eq:wave equation}
u_{xx} = u_{yy}.
\end{equation}
This wave equation can be written in divergence form and regularized as in \eqref{eq:linear complex}.
Let us add the Dirichlet boundary conditions $u(x,y) = 0$ when $y < 0$, as well as $u(x,y) = y$ when $y \ge 0$, where $(x,y) \in \partial \Omega$.
Upon discretizing \eqref{eq:wave equation} with Lagrange $\mathbb P^1$ finite elements, and considering $\varepsilon = 3 h$, one gets the result in Figure \ref{fig:viscosity}.
\begin{figure}[!htp]
\centering
\includegraphics[scale=.25]{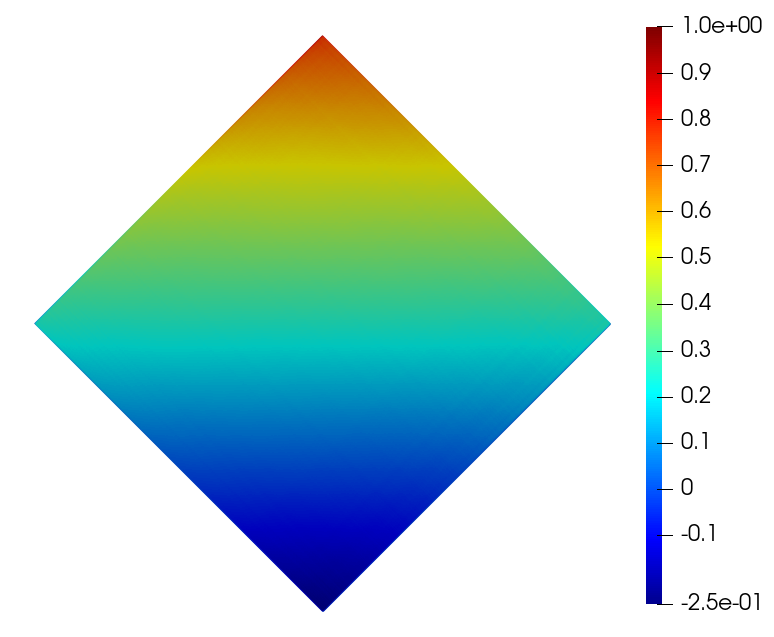}
\caption{Discussion: results of the finite element computation.}
\label{fig:viscosity}
\end{figure}
We observe that the Dirichlet boundary condition $u(x,y) = 0$ when $y < 0$ is ``lost" when $\varepsilon \to 0$.
Note that when $\varepsilon = 1$, for instance, that boundary condition is perfectly respected by the finite element computation.

Since a part of the Dirichlet boundary conditions can be ``lost", it should not be surprising that the numerical results and the experimental results differ when $\varepsilon$ is small.

\section{Conclusion}
This paper has presented the analysis of an approximation of a nonlinear degenerate sign-changing divergence form PDE that models the deformation of a specific type of kirigami called rhombi-slits.
Under appropriate boundary conditions, it proved existence of solutions to \eqref{eq:fix point}.
Then, a numerical method based on Lagrange $\mathbb{P}^1$ finite elements is analyzed and shown to converge towards the solutions of \eqref{eq:fix point}.
Finally, numerical results show the robustness of the method with respect to previous results and experimental results.
Future work includes trying to solve for an imposed $y_\mathrm{eff}$ on the boundary $\partial \Omega_D$ and not a $\xi$.

\section*{Code availability}
The code is available at \url{https://github.com/marazzaf/rhombi_slit.git}.

\section*{Acknowledgment}
The author would like to thank Paul Plucinsky (University of Southern California) and Ian Tobasco (University of Michigan) for fruitful discussions.

The author would also like to thank Paolo Celli (Stony Brook University) for providing him with experimental results, produced for \cite{zheng2022continuum}, and with a \texttt{matlab} code to extract the true deformations from the experiments.

\section*{Funding}
This work is supported by the US National Science Foundation under grant number OIA-1946231 and the Louisiana Board of Regents for the Louisiana Materials Design Alliance (LAMDA).

\bibliographystyle{plain}
\bibliography{bib}

\end{document}